\documentclass[a4paper,11pt]{article}
\usepackage{indentfirst,latexsym,amssymb,amsmath}
\usepackage[dvips]{graphicx}
\usepackage{amsmath,indentfirst,amsfonts,dsfont,amssymb,amsthm,graphicx,enumerate,mathrsfs,color,float}
\usepackage{subfigure}
\newtheorem{theorem}{Theorem}[section]

\newtheorem{lemma}{Lemma}[section]

\newtheorem{definition}{Definition}[section]

\newtheorem{exmple}{Example}[section]
\numberwithin{equation}{section}
\numberwithin{equation}{section}

\begin{document}
\title{Gradient Dynamic Approach to the Tensor Complementarity Problem}
\author{Xuezhong Wang\thanks{E-mail: xuezhongwang77@126.com. School of
Mathematics and Statistics, Hexi University, Zhangye, 734000, P. R. of
China.  This author is supported by the National Natural Science
Foundation of China under grant 11771099.}
\and Maolin Che\thanks{E-mail: chncml@outlook.com and cheml@swufe.edu.cn. School of Economic Mathematics, Southwest University of Finance and Economics, Chengdu, 611130, P. R. of China.  This author is supported by the Fundamental Research Funds for the Central Universities under grant JBK 1801058.} \and Liqun Qi\thanks{ E-mail: maqilq@polyu.edu.hk. Department of Applied Mathematics, the Hong Kong Polytechnic University, Hong Kong. L. Qi is supported by the Hong Kong Research Grant Council (Grant No. PolyU 15302114, 15300715, 15301716 and 15300717)}
\and Yimin Wei\thanks{
Corresponding author (Y. Wei).
E-mail: ymwei@fudan.edu.cn and yimin.wei@gmail.com. School of
Mathematical Sciences and Shanghai Key Laboratory of Contemporary
Applied Mathematics, Fudan University, Shanghai, 200433, P. R. of
China. This author is supported by the National Natural Science
Foundation of China under grant 11771099.}
}
\maketitle

\begin{abstract}
Nonlinear gradient dynamic approach for solving the tensor complementarity problem (TCP) is presented. Theoretical analysis shows that each of the defined dynamical system models ensures the convergence performance.  The computer simulation results further substantiate that the considered dynamical system can solve the tensor complementarity problem (TCP).
\end{abstract}

{\bf Key words}: Tensor complementarity problem; dynamical system; activation function; convergence.\\

{\bf AMS subject classifications:} 15A18, 15A69, 65F15, 65F10, 90C11

\newpage

\section{Introduction}
Let $\mathbb{R}$  be the real field. A tensor can be regarded as a high-order generalization of a matrix, which takes a form
\begin{equation*}
\mathcal{A}=(a_{i_1i_2\dots i_m}),\quad a_{i_1i_2\dots i_m}\in\mathbb{R},\quad 1\leq i_1,i_2,\ldots,i_m\leq n.
\end{equation*}
Such a multi-array $\mathcal{A}$ is said to be an $m$-order $n$-dimensional real tensor. We denote the set of all $m$-order $n$-dimensional real tensor
by $\mathbb{R}^{[m,n]}$.
Let $x\in\mathbb{R}^{n}$, the $n$ dimensional vector $\mathcal{A}x^{m-1}$ is defined as \cite{Qi1}:
\begin{equation}\label{eq2}
 (\mathcal{A}x^{m-1})_{i}=\sum\limits_{i_2,\ldots,i_m}^{n}a_{ii_2,\ldots,i_m}x_{i_2}\ldots x_{i_m},\; i=1,2,\ldots,n,
\end{equation}
where $x_i$ denotes the $i$th component of $x$.

For any $q\in\mathbb{R}^{n} $, we consider the tensor complementarity problem, a special class of nonlinear complementarity problems, denoted by TCP$(\mathcal{A},q)$: finding $x\in\mathbb{R}^{n}$ such that
 \[x\geq 0,\quad\mathcal{A}x^{m-1}+q\geq0,\quad x^{\top}(\mathcal{A}x^{m-1}+q)=0. \]
This is a generalization of the linear complementarity problem. So far many
researchers have paid attention to this topic \cite{Bai2016Global,Che2016Positive,Ding2015P,Du2018,Gowda2016Z,Huang2017Formulating,Luo2017The,Song2015Properties,Song2016Tensor,Xie2017An} because of its applications such
as DNA micro-arrays, communication and $n$-person non-cooperative game \cite{Huang2017Formulating, Luo2017The}. In \cite{Song2015Properties}, Song and Qi showed that TCP$(\mathcal{A}, q)$  has a solution if and only if $\mathcal{A}$ is nonnegative with its diagonal entries being positive. Song and Qi \cite{Song2016Tensor} discussed the solution of TCP$(\mathcal{A}, q)$, when $\mathcal{A}$ is strictly semi-positive.
 Che, Qi and Wei \cite{Che2016Positive} discussed the existence and uniqueness of solution
of TCP$(\mathcal{A}, q)$ with some special tensors. Luo, Qi and Xiu \cite{Luo2017The} obtained the sparsest solutions to TCP$(\mathcal{A}, q)$ with a $\mathcal{Z}$-tensor. Song and Yu \cite{Song2016Properties} obtained global
upper bounds of the solution of the TCP$(\mathcal{A}, q)$ with a strictly semi-positive
tensor. Gowda, Luo, Qi and Xiu \cite{Gowda2016Z} studied the various equivalent conditions for the existence of solution to TCP$(\mathcal{A}, q)$ with a $\mathcal{Z}$-tensor.
Ding, Luo and Qi \cite{Ding2015P} showed the properties of TCP$(\mathcal{A}, q)$ with a $P$-tensor.
Bai, Huang and Wang \cite{Bai2016Global} considered the global uniqueness and solvability
for TCP$(\mathcal{A}, q)$ with a strong $P$-tensor. Wang, Huang and Bai \cite{Wang2016Exceptionally} gave the solvability of TCP$(\mathcal{A}, q)$ with exceptionally regular tensors.

Numerical algorithms for solving tensor complementarity problems have been
proposed recently. Xie, Li and Xu \cite{Xie2017An} presented numerical methods for finding the least solution
to the TCP$(\mathcal{A}, q)$ with a $\mathcal{Z}$-tensor. Liu, Li and Vong \cite{Liu2017Tensor} proposed the modulus equation for TCP$(\mathcal{A}, q)$ and based on this equation, they developed the corresponding nonsmooth Newton's method for solving TCP$(\mathcal{A}, q)$.  Huang and Qi \cite{Huang2017Formulating} proposed a smoothing type algorithm. Han \cite{Han2018} introduced a Kojima-Megiddo-Mizuno type continuation method for solving TCP$(\mathcal{A}, q)$.
Du et al. \cite{Du2018} showed that the tensor absolute equation is equivalent
to a generalized tensor complementarity problem and proposed an inexact Levenberg-Marquardt method for solving the tensor absolute equation. Also, Du and Zhang \cite{Du2018a} gave a mixed integer programming model to solve the TCP$(\mathcal{A}, q)$.

The introduction of dynamic models in optimization started in 1980s
\cite{chua1984nonlinear,hopfield1985neural}. Since then, significant research results have been achieved for various optimization problems,
such as linear programming \cite{zak1995solving}, quadratic programming \cite{bouzerdoum1993neural}, linear complementarity problems \cite{LIZHI19999}, and nonlinear programming \cite{rodriguezvazquez1990nonlinear}. The essence of dynamic approach for optimization is
to establish an energy function (nonnegative).
 The dynamic system is normally in the form of first-order ordinary differential
equations. It is expected that for an initial state, the dynamic system will approach its static
state (or equilibrium point) which corresponds the solution of the underlying optimization problem.
An important requirement is that the energy function decreases monotonically as the dynamic system
approaches an equilibrium point.

The gradient dynamical system (GDS) has now been regarded as a powerful alternative for online computation \cite{Zhang2006A},
linear complementarity problems \cite{LIZHI19999} and nonlinear complementarity problems \cite{LIAO2001},  in view of its high speed processing nature and its convenience of hardware implementation in practical applications \cite{Feng2006Gradient,Ramezani2013Nonlinear}. To effectively solve the linear complementarity problem and nonlinear complementarity problems, linear gradient dynamical system (LGDS) is thus obtained \cite{LIZHI19999,LIAO2001}. Note
that the LGDS with application to online linear (nonlinear) complementarity problems solving have been investigated by the previous work \cite{LIZHI19999,LIAO2001}.
The existence and the convergence of the trajectory of the dynamical system are addressed in detail in \cite{LIZHI19999,LIAO2001}. In addition, Liao, Qi and Qi \cite{LIAO2001} also explore the stability properties, such as the stability in the sense of Lyapunov, the asymptotic stability and the exponential stability, for the dynamical system model.
However, to the best of our knowledge, there exists few research results on solving the TCP$(\mathcal{A}, q)$ via the nonlinear gradient dynamical system (NGDS). Motivated by this
reason, we thus design, propose and investigate different NGDS for solving TCP$(\mathcal{A}, q)$ by defining error-monitoring functions. It is theoretically proved that defined  NGDS converge to the theoretical solution.
Through illustrative computer-simulation examples, the efficacy and the superiority of the proposed dynamical system model for online computation of the TCP$(\mathcal{A}, q)$ with $\mathcal{A}x^{m-1}+q$ is a $P$-function is well-verified.

The main contributions of the paper are listed as follows.

(1) One type of NGDS for solving the TCP$(\mathcal{A}, q)$ with $\mathcal{A}x^{m-1}+q$ for any $q\in \mathbb{R}^{n}$ is a $P$-function are presented;

(2) Theoretical analysis shows the convergence of the presented gradient neural networks to the theoretical solution of the TCP$(\mathcal{A}, q)$ with $\mathcal{A}x^{m-1}+q$ for any $q\in \mathbb{R}^{n}$ is a $P$-function;

(3) Computer simulation results via illustrative examples are presented, compared and
discussed. Generated numerical results comparatively substantiate that the NGDS with nonlinear activation function are much more efficient in solving the TCP$(\mathcal{A}, q)$ with $\mathcal{A}x^{m-1}+q$ for any $q\in \mathbb{R}^{n}$ is a $P$-function, as compared to the NGDS with linear activation function proposed in this paper.

This paper is organized as follows. In Section \ref{sec2}, we recall some preliminary definitions and
results. Dynamical system models with different nonlinear activation functions for online solution of the TCP$(\mathcal{A}, q)$ with $\mathcal{A}x^{m-1}+q$  for any $q\in \mathbb{R}^{n}$ is a $P$-function are presented in Section \ref{sec3}. Convergence properties of the presented dynamical system models will be discussed in Section \ref{sec4}.
Illustrative numerical examples are presented in Section \ref{Examples}.

\section{Preliminaries}\label{sec2}
Han \cite{HAN201749} gave a method to partially symmetrize tensor $\mathcal{A}=(a_{i_1i_2\ldots i_m})\in \mathbb{R}^{[m,n]}$ with respect to the indices $i_2\ldots i_m$, which will be used in sequel. In detail, the partially symmetrized tensor $\mathcal{\widehat{A}}=(\widehat{a}_{i_1i_2\ldots i_m})$ as follows
\[\widehat{a}_{i_1i_2\ldots i_m}=\frac{1}{(m-1)!}\sum\limits_{\pi}a_{i_1\pi(i_2\ldots i_m)},\]
where the sum is over all the permutations $\pi(i_2\ldots i_m)$. For any $\mathcal{A}\in \mathbb{R}^{[m,n]}$, we can get a partially symmetrized tensor $\mathcal{\widehat{A}}\in \mathbb{R}^{[m,n]}$ such that $\mathcal{A}x^{m-1}=\mathcal{\widehat{A}}x^{m-1}$, by an averaging procedure.

\subsection{Function tensors and matrices}
We first recall the definitions of $P$-matrix, $P$-tensor and $P$-function as follows.
\begin{lemma}{\bf (\cite{Berman1994Plemmons})}\label{pmat}
Let $A\in \mathbb{R}^{n\times n}$, then
 $A$ is called a $P$-matrix if all the principal minors of $A$ are positive.

\end{lemma}

Throughout the paper, we assume that $F:\mathbb{R}^n\rightarrow \mathbb{R}^n$ is a continuously differentiable function.
\begin{definition}{\bf (\cite{Facchinei2003Finite})}\label{pfunction}
A function $F:K\subseteq \mathbb{R}^n\rightarrow \mathbb{R}^n$ is called a
 $P$-function on $K$ if for all $x,y\in \mathbb{R}^n$ with $x\neq y$, it holds that
\begin{equation}\label{peq1}
 \max_{i}(x_i-y_i)[F_{i}(x)-F_{i}(y)]>0.
\end{equation}
\end{definition}
\begin{lemma}{\bf (\cite{Facchinei2003Finite})}\label{plema}
$F$:\;$\Omega\supset K\subseteq \mathbb{R}^n\rightarrow \mathbb{R}^n$ be continuously differentiable on the
open set $\Omega$ containing the set $K$. $F$ is a $P$-function on $K$ if and only if Jacobian matrix $F'(x)$ is a $P$-matrix for all $x\in K$.


\end{lemma}
\begin{definition}{\bf (\cite{Bai2016Global,Liu2017Tensor})}\label{ptensor}
Let $\mathcal{A}\in\mathbb{R}^{[m,n]}$, then $\mathcal{A}$ is called\\
{\rm (i)} a $P$-tensor, if for each nonzero $x\in \mathbb{R}^n$ there exists index $i$ such that
\begin{equation}\label{peq2}
 x_i(\mathcal{A}x^{m-1})_i> 0;
\end{equation}
{\rm (ii)} a strong $P$-tensor, if and only if $\mathcal{A}x^{m-1}$ is a $P$-function;\\
{\rm (iii)} a strong strictly semi-positive tensor if $\mathcal{A}x^{m-1}+q$ is a $P$-function in $\mathbb{R}_{+}^{n}$ for any $q\in \mathbb{R}^n$, where $\mathbb{R}_{+}^{n}$ denotes the set of all the nonnegative vectors.
\end{definition}
From the definitions of the $P$-tensor, the strong $P$-tensor and the strong strictly semi-positive tensor, it is easy to see that every strong $P$-tensor must be a $P$-tensor, a strong strictly semi-positive tensor is a $P$-tensor. In additional, we can obtain if for any $q\in \mathbb{R}^n$, $\mathcal{A}x^{m-1}+q$ is a $P$-function, then tensor $\mathcal{A}$ is a $P$-tensor, furthermore, we have Jacobian matrix $(\mathcal{A}x^{m-1})'$ is a $P$-matrix.

Bai, Huang, and Wang \cite{Bai2016Global} proved that a TCP possesses the
global uniqueness and solvability property if the tensor is a strong $P$-tensor.  Liu {\it et al.} \cite{Liu2017Tensor} showed that a TCP possesses the
global uniqueness and solvability property if the tensor is strong strictly semi-positive tensor.  We summarize their results in the following theorem.
\begin{lemma}{\bf (\cite{Bai2016Global,Liu2017Tensor})}\label{pexist} Let $\mathcal{A}\in \mathbb{R}^{[m,n]}$,\\
{\rm (i)} if $\mathcal{A}$ is a $P$-tensor, then, for any $q\in \mathbb{R}^{n}$, the solution set of $TCP(\mathcal{A},q)$ is nonempty and compact;\\
{\rm (ii)} if $\mathcal{A}$ is a strong $P$-tensor, then $TCP(\mathcal{A},q)$ has the global uniqueness and solvability property;\\
{\rm (iii)} if $\mathcal{A}$ is a strong strictly semi-positive tensor, then $TCP(\mathcal{A},q)$ has the global uniqueness and solvability property.
\end{lemma}
\subsection{The Fischer-Burmeister NCP function}
Below, we introduce the classical nonlinear complementarity problem (NCP). It will be
shown that the tensor complementarity problem TCP$(\mathcal{A},q)$ is a
special kind of nonlinear complementarity problem.
\begin{definition}{\bf (\cite{Facchinei2003Finite})}\label{ncp}
Given a mapping $F$ : $\mathbb{R}^n\rightarrow \mathbb{R}^n$, the nonlinear complementarity problem,
denoted by NCP(F), is to find a vector $x\in \mathbb{R}^n $ satisfying
\begin{equation}\label{eqncp}
  x\geq 0,\quad F(x)\geq 0,\quad x^{\top}F(x)=0.
\end{equation}
\end{definition}
Note that if $F(x)=\mathcal{A}x^{m-1}+q$, then NCP reduces to TCP$(\mathcal{A},q)$.

Many solution methods developed for NCP or related problems are based on
reformulating them as a system of equations using so-called NCP-functions. Here,
a function $\phi:\;\mathbb{R}^2\rightarrow \mathbb{R}$ is called an NCP-function if
\begin{equation}\label{ncpfun}
 \phi(a,b)=0\Leftrightarrow ab=0,\quad a\geq 0, \quad b\geq 0.
\end{equation}
Here, we use the following Fischer-Burmeister NCP-function \cite{ Fischer1992A},
\[\phi(a,b)=\sqrt{a^2+b^2}-a-b,\]
which are widely used in nonlinear complementarity problems.

\begin{lemma}\label{lefb}
The square of $\phi(a,b)$ is continuously differentiable; $\phi(a,b)$ is twice continuously differentiable everywhere except at the origin; but it is strongly semismooth at the origin.
\end{lemma}
By a simple calculation procedure, we have
\[\partial \phi(a,b)=\left\{\begin{array}{lc}
                             \left(\frac{a}{\sqrt{a^2+b^2}}-1,\frac{b}{\sqrt{a^2+b^2}}-1\right), & a^2+b^2\neq 0, \\
                             (\alpha-1,\beta-1)\;\text{with}\; \alpha^2+\beta^2\leq 1, & a^2+b^2=0.
                           \end{array}\right.
\]
From Lemma \ref{lefb}, we have $\partial \phi(a,b)$ is a singleton except at the region.
\subsection{Activation function}\label{subsec2.3}

The matrix-valued activation function $\mathcal{F}(E)$, $E=(e_{ij})$, is defined as $(f(e_{ij}))$, $i,j=1,2,\ldots,n$, where $f(\cdot)$ is a scalar-valued monotonically-increasing odd function. The following real-valued linear and nonlinear odd and monotonically increasing functions $f(\cdot)$ are widely used.\\

Linear function
  $$f_{\rm{lin}}(x)=x;$$

Bipolar-sigmoid function
$$
 f_{\rm{bs}} (x,q)=\frac{1+\exp(-q)}{1-\exp(-q)}\frac{1-\exp(-qx)}{1+\exp(-qx)},\;q>2;
$$

Power-sigmoid function
$$
 f_{\rm{ps}}(x,p,q)=\left\{\begin{array}{cl}
                         x^p, & \mathrm{if }\ |x|\geq1 \\
                         \frac{1+\exp(-q)}{1-\exp(-q)}\frac{1-\exp(-qx)}{1+\exp(-qx)},    &  \mathrm{otherwise}
                       \end{array},\right.\;q\geq2,\;p\geq3;
$$

Smooth power-sigmoid function
$$
f_{\rm{sps}}(x,p,q)=\frac{1}{2}x^{p}+\frac{1+\exp(-q)}{1-\exp(-q)}\frac{1-\exp(-qx)}{1+\exp(-qx)},\;p\geq3,\;q>2.
$$

In general, any monotonically increasing odd activation function $f(\cdot)$  can be used for
the construction of the dynamical system.
As it was shown in \cite{Li2013Accelerating,Zhang2009Global}, the
convergence rate can be thoroughly improved by an appropriate activation function. So far,
the influence of various nonlinear activation functions was investigated for different dynamical system models. We
investigate this scenario on several dynamical system models which are introduced in this paper.

\section{Nonlinear dynamical system methods}\label{sec3}
In this section, the error monitoring function is designed for deriving a gradient dynamical system (GDS). Specifically, by defining different EFs, different GDSs can be obtained for online solution of the  TCP$(\mathcal{A},q)$. We construct nonlinear gradient dynamical system models, called NGDS, and consider their convergence.

Since  TCP$(\mathcal{A},q)$ can be equivalently reformulated as finding a solution of the following equation:
\begin{equation}\label{ERRf}
 \Phi(x)=\left(
           \begin{array}{c}
             \phi(x_1,(\mathcal{A}x^{m-1}+q)_1) \\
                 \phi(x_2,(\mathcal{A}x^{m-1}+q)_2) \\
             \vdots \\
            \phi(x_n,(\mathcal{A}x^{m-1}+q)_n) \\
           \end{array}
         \right)=0.
\end{equation}
We note that $\Phi(x)$ is locally Lipschitz continuous everywhere, so that Clarke's \cite{Clarke1983Optimization} generalized Jacobian
$\partial\Phi(x)$ is well defined at any point.

Thus, we can define the error monitoring function as
\[\varepsilon(t)=\varepsilon(x(t))=\frac{1}{2}\|\Phi(x(t))\|_{2}^2.\]
The function $\varepsilon(x(t))$ is continuously differentiable \cite{facchinei1997a}, which follows from the semi-smoothness of $\Phi(x)$.

In order to force $\varepsilon(t)$ to converge to zero, the
negative of the gradient (i.e., -$\partial(\varepsilon(t))/\partial x$) is used as the
descent direction, which leads to the so-called GDS design formula in the form of a first-order differential
equation:
\begin{equation}\label{GNNF}
\frac{\mathrm{d}x}{\mathrm{d}t}=-\gamma\frac{\partial(\varepsilon(t))}{\partial x}=-\gamma V^{\top}\Phi(x),
\end{equation}
where  $V\in\partial\Phi(x)$ and $\gamma$ is a positive scaling constant. Note that $\gamma$ corresponds to the reciprocal of a capacitance parameter, of
which the value should be set as large as the hardware
would permit, or appropriately large for modeling and experimental purposes. The dynamic equation (\ref{GNNF}) will be simply termed the LGDS model.


Following the principle of nonlinear activation in the LGDS model defined in Section \ref{subsec2.3}, the conventional LGDS model (\ref{GNNF}) can be improved
into the following nonlinear GDS model by exploiting a nonlinear activation function array
$\mathcal{F}(\cdot)$:
\begin{equation}\label{GNN-I2}
\frac{\mathrm{d}x}{\mathrm{d}t}=-\gamma V^{\top}\mathcal{F}(\Phi(x)),
\end{equation}
where $\mathcal{F}(\cdot)$ denotes a matrix-valued activation function. The dynamic equation (\ref{GNN-I2}) will be simply termed the NGDS model.

We give the precise definition of $V$ which is necessary for the implementation of our model.
\begin{lemma}{\bf(\cite{deluca1996a})}
Any $V$ with the following structure is an element of $\partial\Phi(x)$
\[V=D_{a}(x)+(m-1)D_{b}(x)\mathcal{\widehat{A}}x^{m-2},\]
where $D_{a}(x)=$diag$(a_1(x),\ldots,a_n(x))$, $D_{b}(x)=$diag$(b_1(x),\ldots,b_n(x))$ are diagonal matrices whose $i$th diagonal elements are given by
\[a_i(x)=\frac{x_i}{\sqrt{x_i^2+(\mathcal{A}x^{m-1}+q)_i^2}}-1,\quad b_i(x)=\frac{(\mathcal{A}x^{m-1}+q)_i}{\sqrt{x_i^2+(\mathcal{A}x^{m-1}+q)_i^2}}-1\]
if $(x_i,(\mathcal{A}x^{m-1}+q)_i)\neq (0,0)$, and by
\[a_i(x)=\alpha_i-1,\quad b_i(x)=\beta_i-1\]
for every $(\alpha_i,\beta_i)\in \mathbb{R}^2$ such that $\alpha_i^2+\beta_i^2\leq 1$ if $(x_i,(\mathcal{A}x^{m-1}+q)_i)= (0,0)$.
\end{lemma}

\section{Stability analysis}\label{sec4}
In this section, we address the stability issues on the dynamical system (\ref{GNN-I2}) to the solution of the  TCP$(\mathcal{A},q)$.

For given $x_{*}\in \mathbb{R}^{n}$ with $x_{*}$ is a solution of the  TCP$(\mathcal{A},q)$, define $\delta=\min\|x_{*}-u\|_2$ for any $u\in \mathbb{R}^{n}$. Thus, for any $0<\widehat{\delta}\leq\delta$, we define a neighbourhood of $x_{*}$ as
\begin{equation}\label{BQOP:equation16}
\mathbb{B}(x_{*};\widehat{\delta}):=\{x\mid\|x-x_{*}\|_2\leq\widehat{\delta}\},
\end{equation}
where $x\in\mathbb{R}^n$.

Now we recall some stability results from \cite{Zabczyk2015Mathematical} on the following differential equation:
\begin{equation}\label{eqnnew}
 \frac{\mathrm{d}x(t)}{\mathrm{d}t}=f(x(t)), \quad x(t_0)\in\mathbb{R}^n.
\end{equation}

The following classical results on the existence and uniqueness of the solution to (\ref{eqnnew}) hold.
\begin{definition}{\bf(\cite{Zabczyk2015Mathematical})} \label{def4.1}
Let $x(t)$ be a solution of (\ref{eqnnew}). An isolated equilibrium point $x^*$ is Lyapunov stable if for any $x(t_0)$ and any scalar $\epsilon>0$ there exists a $\hat{\delta}>0$
so that if $x(t_0)\in \mathbb{B}(x_{*},\widehat{\delta})$ then $\|x(t)-x_*\|_2<\epsilon$ for $t\geq t_0$.
\end{definition}
\begin{definition}{\bf(\cite{Zabczyk2015Mathematical})} \label{def4.2}
 An isolated equilibrium point $x_*$ is said to be asymptotic stable
if in addition to being a Lyapunov stable it has the property that
$x(t)\rightarrow x_*$ as $t\rightarrow +\infty $, if $x(t_0)\in \mathbb{B}(x_{*},\hat{\delta})$.
\end{definition}

Then we focus on a particular case where the equilibrium point is isolated. Let $S$ denote the solution set of the  TCP$(\mathcal{A},q)$ and $x\in S$  implies $\Phi(x)=0$. Togather with   $\mathcal{F}(\cdot)$ is odd and monotonically increasing function, we have $\mathcal{F}(\Phi(x))=0$, consequently, $\frac{\mathrm{d}x}{\mathrm{d}t}=0$. Hence, we have the following result.
\begin{theorem}\label{theq}
Every solution to the TCP$(\mathcal{A},q)$ is an equilibrium point of the dynamical system (\ref{GNN-I2}).
Conversely; if $x\in \mathbb{R}^n$ is an equilibrium of (\ref{GNN-I2}) and for any $q\in \mathbb{R}^n$, $\mathcal{A}x^{m-1}+q$ is a $P$-function, then $x\in S$.
\end{theorem}
\begin{proof}
We only need to address the second part of the Theorem. Since  $\mathcal{A}x^{m-1}+q$ is a continuously differentiable $P $-function, it's Jacobian matrix is a $P $-matrix for all $x\in \mathbb{R}^{n}$. Analogy to the proof of  \cite[Corollary  4.4]{deluca1996a}, we have  a stationary point of (\ref{GNN-I2}) is a solution to the TCP$(\mathcal{A},q)$.
We complete our proof.
\end{proof}

Note that if $\mathcal{A}$ is a strong $P$-tensor (strong strictly semi-positive tensor), then $TCP(\mathcal{A},q)$ has the global uniqueness and solvability property. Together with a strong $P$-tensor (strong strictly semi-positive tensor) is a $P$-tensor, we obtain the following results.
\begin{theorem}\label{strongp}
Let $\mathcal{A}$ be a strong $P$-tensor (strong strictly semi-positive tensor) and $x$ be a solution of the TCP$(\mathcal{A},q)$, then $x$ is an unique equilibrium point of the dynamical system (\ref{GNN-I2}).
\end{theorem}
\begin{proof}
Since $\mathcal{A}$ is a strong $P$-tensor (strong strictly semi-positive tensor), then $TCP(\mathcal{A},q)$ has the global uniqueness and solvability property. That is if $x$ is a solution of the TCP$(\mathcal{A},q)$, then $x$ is a unique solution of the TCP$(\mathcal{A},q)$. From the Theorem \ref{theq}, we get $x$ is an unique equilibrium point of the neural network (\ref{GNN-I2}). The proof is thus completed.
\end{proof}
We have the following result of the convergence of the NGDS model.
\begin{theorem} \label{theorem4.1}
Given $\widehat{\delta}>0$ and $x_{*}$ be a solution of the TCP$(\mathcal{A},q)$, if nonzero vector $x(0)\in \mathbb{B}(x_{*};\widehat{\delta})$ and  $\mathcal{A}x(t)^{m-1}+q$ is a $P$-function for any $q\in \mathbb{R}^n$, then the state $x(t)$ of the {\em NGDS} model (\ref{GNN-I2}), starting from the initial state $x(0)\in \mathbb{B}(x_{*};\widehat{\delta})$, converges to the solution $x_*\in \mathbb{R}^{n}$ of the  TCP$(\mathcal{A},q)$.
\end{theorem}
\begin{proof} We construct the following Lyapunov function:
\begin{equation}\label{lyap}
  L(t)=\varepsilon(x(t))=\frac{1}{2}\|\Phi(x(t))\|_{2}^2\geq 0,
\end{equation}
with its time derivative being
\begin{equation}\label{eq4.2}
\begin{array}{ccl}
 \frac{\mathrm{d}L(t)}{\mathrm{d}t}&= &\frac{1}{2}\frac{\mathrm{d}}{dt} \mathrm{Tr}\left(\Phi(x(t))^{\top}\Phi(x(t))\right)\\
 &= &\left(\Phi(x(t))^{\top}\frac{\mathrm{d}\Phi(x(t)))}{\mathrm{d}t}\right)\\
  &=&\left(\Phi(x(t))^{\top}V\frac{\mathrm{d}x}{\mathrm{d}t}\right)\\
    &=&-\gamma\left(\Phi(x(t))^{\top}VV^{\top}\mathcal{F}(\Phi(x(t)))\right)\\
 &=&-\gamma\mathrm{Tr}\left(VV^{\top}\Phi(x(t))^{\top}\mathcal{F}(\Phi(x(t)))\right).\\
 \end{array}
 \end{equation}
Since $V$ is a $P$-matrix, which follows from $\mathcal{A}x^{m-1}+q$ is a $P$-function for any $q\in \mathbb{R}^n$, then, $VV^{\top}$ is a symmetric positive definite matrix, and then, we have

\[\begin{array}{cll}
     \lambda_{\min}\mathrm{Tr}\left(\Phi(x(t))^{\top}\mathcal{F}(\Phi(x(t)))\right) & \leq& \mathrm{Tr}\left(VV^{\top}\Phi(x(t))^{\top}\mathcal{F}(\Phi(x(t)))\right) \\
     & \leq & \lambda_{\max}\mathrm{Tr}\left(\Phi(x(t))^{\top}\mathcal{F}(\Phi(x(t)))\right),
  \end{array}
\]
where $ \lambda_{\min}$ and $ \lambda_{\max}$ denote the smallest eigenvalue and the largest eigenvalue of matrix $VV^{\top}$, respectively.  $\mathrm{Tr}(A)$ is the trace of the matrix $A$.

Since the scalar-valued function $f(\cdot)$ is an odd and monotonically increasing function, it immediately follows $f(-x)=-f(x)$ and
\[f(x)\left\{\begin{array}{cl}
               >0 ,& \mathrm{if}\; x>0, \\
               =0, & \mathrm{if} \;x=0, \\
               <0 ,& \mathrm{if} \;x<0,
             \end{array}
\right. \]
which implies
\[xf(x)\left\{\begin{array}{cl}
               >0, & \mathrm{if}\; x\neq 0, \\
               =0,& \mathrm{if} \;x=0.
             \end{array}
\right. \]
It follows that
\[\mathrm{Tr}\left(\Phi(x(t))^{\top}\mathcal{F}(\Phi(x(t)))\right)\left\{\begin{array}{cl}
               >0, & \mathrm{if}\; \Phi(x(t))\neq 0, \\
               =0, & \mathrm{if} \;\Phi(x(t))=0.
             \end{array}
\right.\]
Due to the fact that the design parameter satisfies $\alpha>0$, in view of (\ref{eq4.2}), it follows that
\[\frac{\mathrm{d}L_1(t)}{\mathrm{d}t}\left\{\begin{array}{cl}
              <0, & \mathrm{if}\; \Phi(x(t))\neq 0, \\
               =0, & \mathrm{if} \;\Phi(x(t))=0.
             \end{array}
\right.\]

%
By the Lyapunov theory,  $\Phi(x(t))$ can converge to zero; or, equivalently speaking, state $x(t)$ of NGDS model (\ref{GNN-I2}) is asymptotic stable at one of solutions $x_{*}$ with $\Phi(x_{*})=0$ starting from an initial state $x(0)\in \mathbb{B}(x_{*};\widehat{\delta})$. The proof is thus complete.
\end{proof}

We summarize the convergence result of the NGDS  method (\ref{GNN-I2}) when $\mathcal{A}$
is strong $P$-tensor (strong strictly semi-positive tensor) in the following theorem.
\begin{theorem} \label{theorem4.2}
Suppose that $\mathcal{A}\in \mathbb{R}^{[m,n]}$ is a strong $P$-tensor (strong strictly semi-positive tensor) and $q\in \mathbb{R}^n$. The state $x(t)$
of the {\em NGDS} model (\ref{GNN-I2}), starting from an arbitrary initial state $x(0)\in \mathbb{R}^{n}$, converges to the unique solution $x_*\in \mathbb{R}^{n}$ of the TCP$(\mathcal{A},q)$.
\end{theorem}
\begin{proof}
Since $\mathcal{A}$ is a strong $P$-tensor (strong strictly semi-positive tensor), from Theorem \ref{strongp}, we know $TCP(\mathcal{A},q)$, for given $q\in \mathbb{R}^{n}$, has the global uniqueness and solvability property. Together with the proof of Theorem \ref{theorem4.1}, we can show that the state $x(t)$
of the { NGDS} model (\ref{GNN-I2}), starting from an arbitrary initial state $x(0)\in \mathbb{R}^{n}$, converges to the unique solution $x_*\in \mathbb{R}^{n}$ of the TCP$(\mathcal{A},q)$. We complete our proof.
\end{proof}
\section{Numerical examples}\label{Examples}
In this section, some computer-simulation examples are
demonstrated to verify the efficacy and the superiority of
the proposed neural network models. We apply the NGDS to the TCP$(\mathcal{A},q)$.

All computations are carried out in Matlab Version 2014a, which has a unit roundoff $ 2^{-53}\approx 1.1\times 10^{-16}$, on a laptop with Intel Core(TM) i5-4200M CPU (2.50GHz) and 7.89GB RAM.

\begin{exmple} \label{eg1}
Consider the tensor $\mathcal{A}\in \mathbb{R}^{[4,2]}$, from \cite{Bai2016Global,Liu2017Tensor}, defined by:
\[a_{1111}=1,\;a_{1222}=-1,\;a_{1122}=1,\;a_{2222}=1,\;a_{2111}=-1,\;a_{2211}=1,\]
and $a_{i_1i_2i_3i_4}=0$ otherwise. This tensor is a $P$-tensor, but not a strong $P$-tensor \cite{Bai2016Global}. However, it is a strong strictly semi-positive tensor \cite{Liu2017Tensor}.
\end{exmple}
In our tests, we take different vectors $q\in \mathbb{R}^2$ and initial vector as $x_0={\rm{rand}}(2,1)$. Trajectories of state variables corresponding to the LGDS with $\gamma=10$ are shown in Figure \ref{figure 1} (a), (c), (e) and Figure \ref{figure 2} (a) and (c), respectively.
Residual errors
\begin{equation}\label{eq6.2}
 \mathrm{Res}=\|\Phi(x)\|_2,
\end{equation}
\begin{figure}[H]
\centering
\subfigure[$\gamma=10$, $q=(-5,-3)^{\top}$.]{\includegraphics[width=3in, height=1.8in]{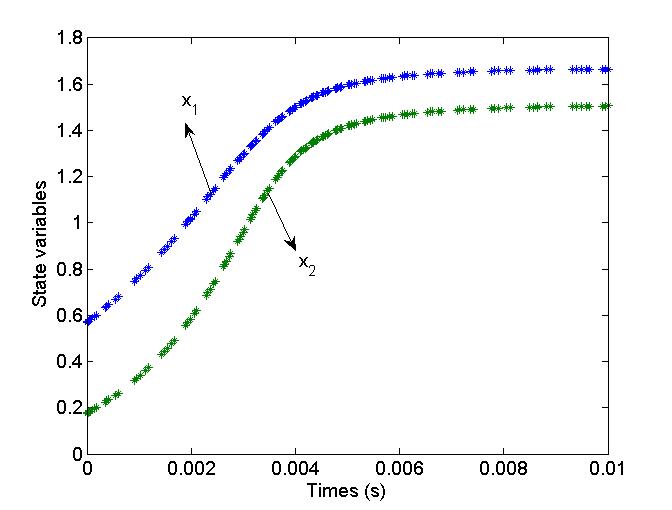}}
\subfigure[$\gamma=1\times 10^6$, $q=(-5,-3)^{\top}$.]{\includegraphics[width=3in, height=1.8in]{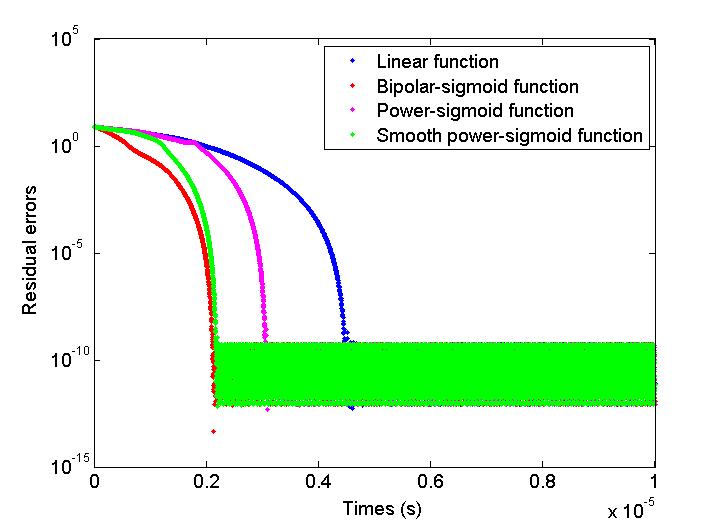}}\\
\subfigure[$\gamma=10$, $q=(-5,3)^{\top}$.]{\includegraphics[width=3in, height=1.8in]{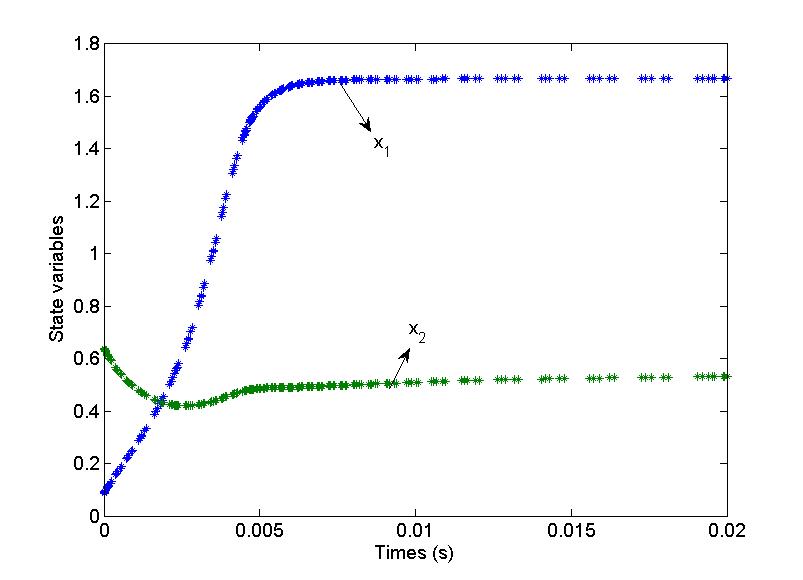}}
\subfigure[$\gamma=1\times 10^6$, $q=(-5,3)^{\top}$.]{\includegraphics[width=3in, height=1.8in]{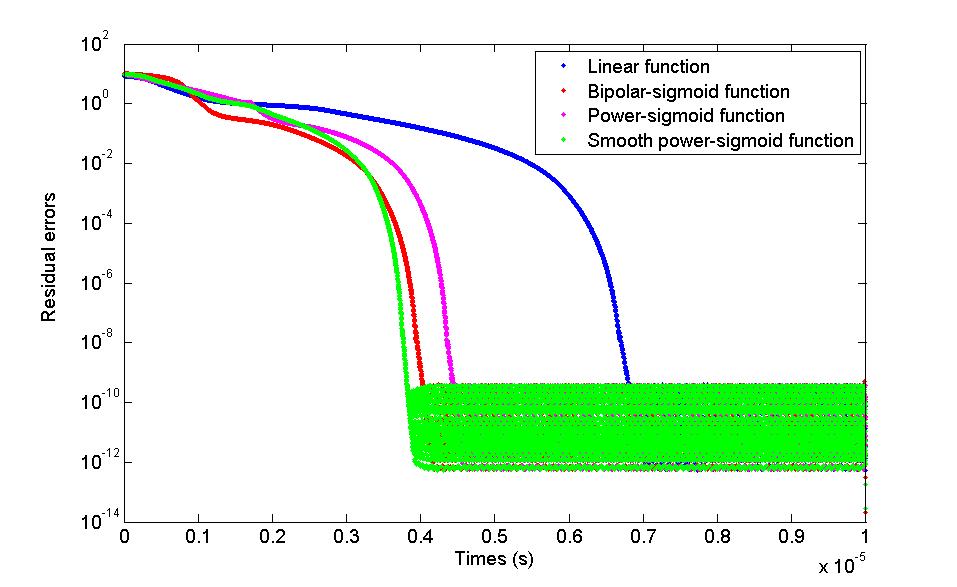}}\\
\subfigure[$\gamma=10$, $q=(5,3)^{\top}$.]{\includegraphics[width=3in, height=1.8in]{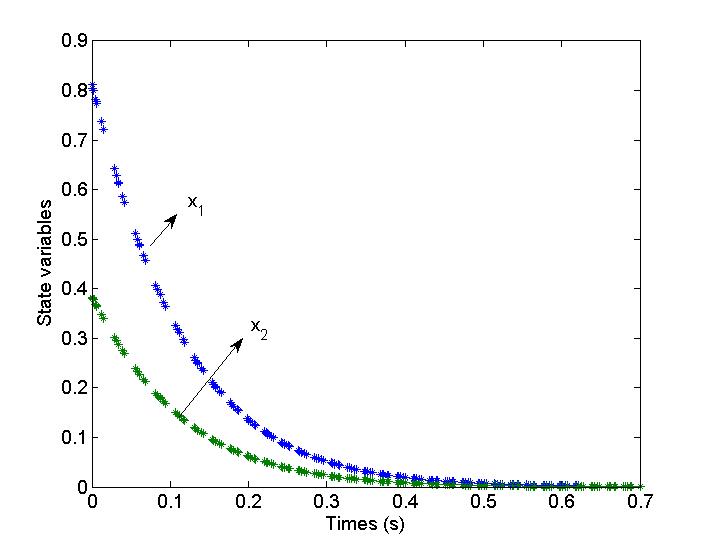}}
\subfigure[$\gamma=1\times 10^6$, $q=(5,3)^{\top}$.]{\includegraphics[width=3in, height=1.8in]{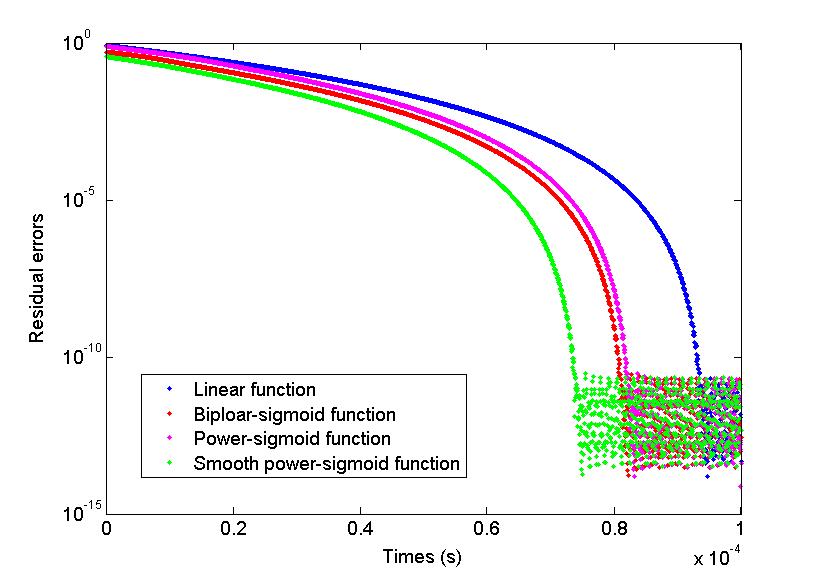}}
\caption{Trajectories of state variables and residual errors of the NGDSs for Example \ref{eg1}.}\label{figure 1}
\end{figure}

derived by employing the NGDS with $\gamma=1\times 10^6$  shown in Figures \ref{figure 1} (b), (d), (f) and Figure \ref{figure 2} (b) and (d), respectively, where  blue stars indicate that $f(\cdot)$ is the linear function, red stars  correspond to the bipolar-sigmoid function $f_{bs}(x,5)$,
pink stars denote that power-sigmoid function $f_{ps}(x,3,5)$  and green stars are generated using $f_{sps}(x,3,7)$ as the smooth power-sigmoid function.
\begin{figure}[!h]
\centering
\subfigure[$\gamma=10$, $q=(2,-3)^{\top}$.]{\includegraphics[width=3in, height=2.2in]{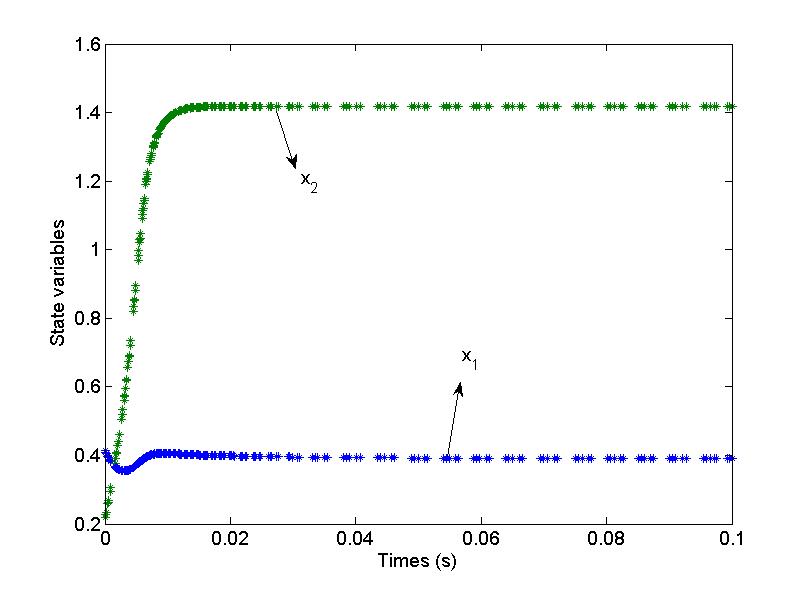}}
\subfigure[$\gamma=1\times 10^6$, $q=(2,-3)^{\top}$.]{\includegraphics[width=3in, height=2.2in]{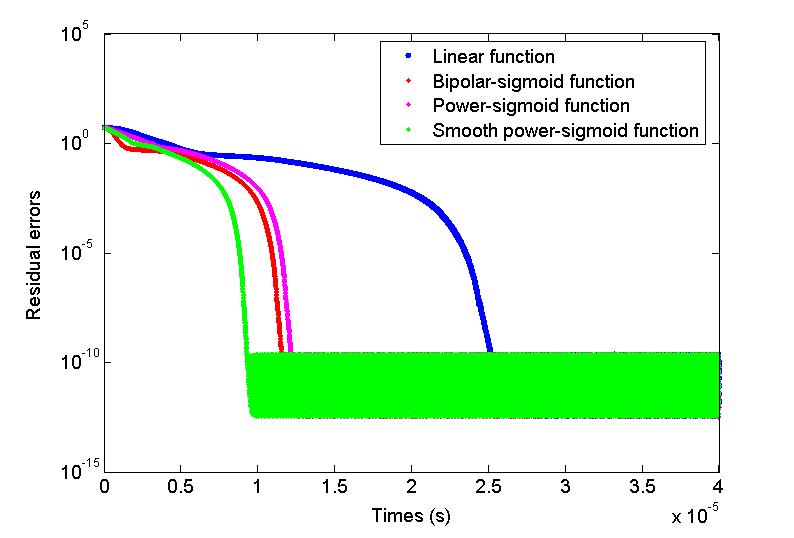}}\\
\subfigure[$\gamma=10$, $q=(0,-5)^{\top}$.]{\includegraphics[width=3in, height=2.2in]{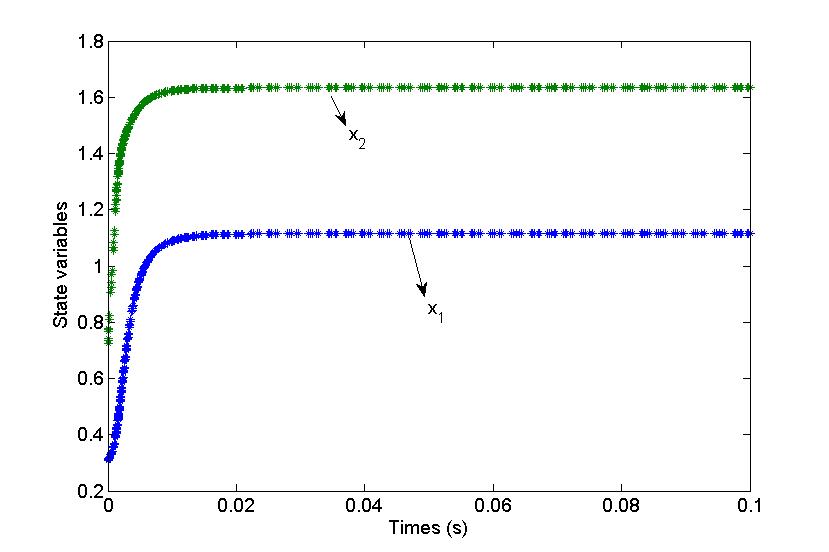}}
\subfigure[$\gamma=1\times 10^6$, $q=(0,-5)^{\top}$.]{\includegraphics[width=3in, height=2.2in]{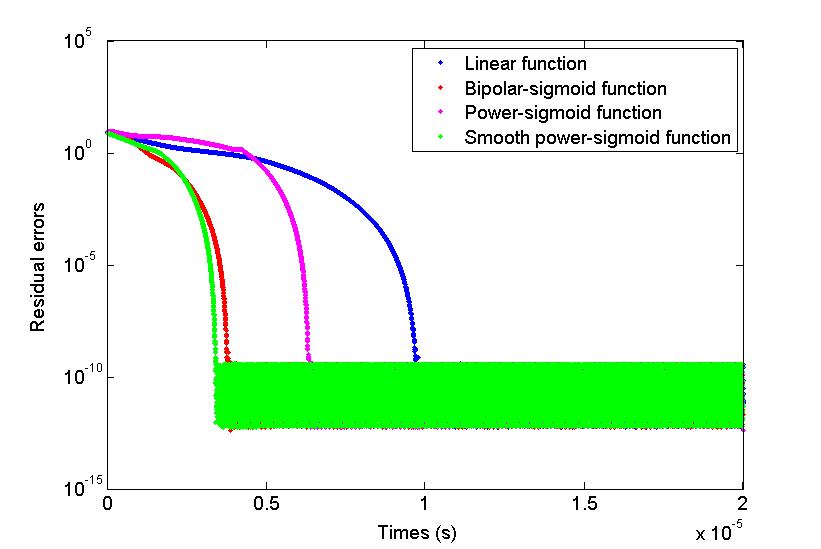}}\\
\caption{Trajectories of state variables and residual errors of the NGDSs for Example \ref{eg1}.}\label{figure 2}
\end{figure}
\begin{exmple} \label{eg2}
Let $\mathcal{A}\in \mathbb{R}^{[5,3]}$ be defined by $a_{kkkkk}=k$ for $k=1,2,3$ and
 $a_{i_1i_2i_3i_4i_5}=0$ otherwise. This tensor is both a strong $P$-tensor \cite{Bai2016Global} and a strong strictly semi-positive \cite{Han2018}.
\end{exmple}
For different vectors $q\in \mathbb{R}^3$, we take  initial vector as $x_0={\rm{rand}}(3,1)$. Trajectories of state variables corresponding to the LGDS with different $\gamma$ are shown in Figure \ref{figure 3} (a), (c), (e) and Figure \ref{figure 4} (a), (c) and (e), respectively.
Residual errors (\ref{eq6.2})
derived by employing the NGDS with $\gamma=1\times 10^6$  shown in Figures \ref{figure 3} (b), (d), (f) and Figure \ref{figure 4} (b), (d) and (f), respectively.

\begin{figure}[H]
\centering
\subfigure[$\gamma=1000$, $q=(1,2,3)^{\top}$.]{\includegraphics[width=3in, height=2.2in]{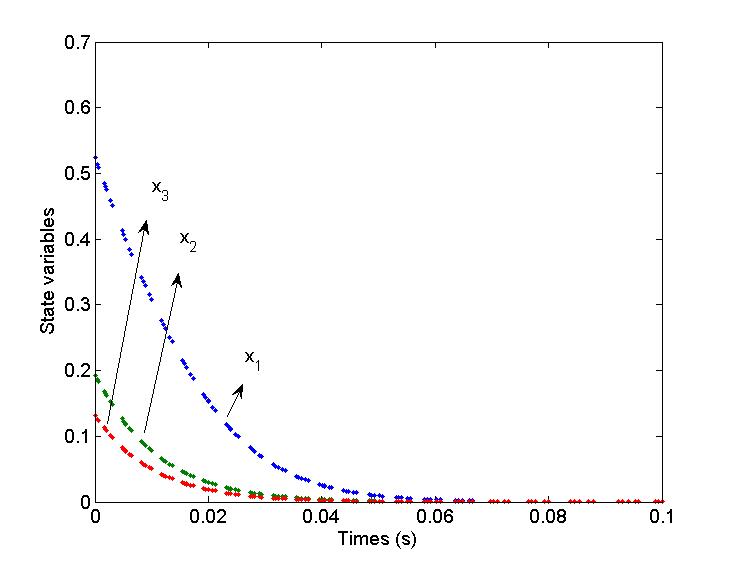}}
\subfigure[$\gamma=1\times 10^6$, $q=(1,2,3)^{\top}$.]{\includegraphics[width=3in, height=2.2in]{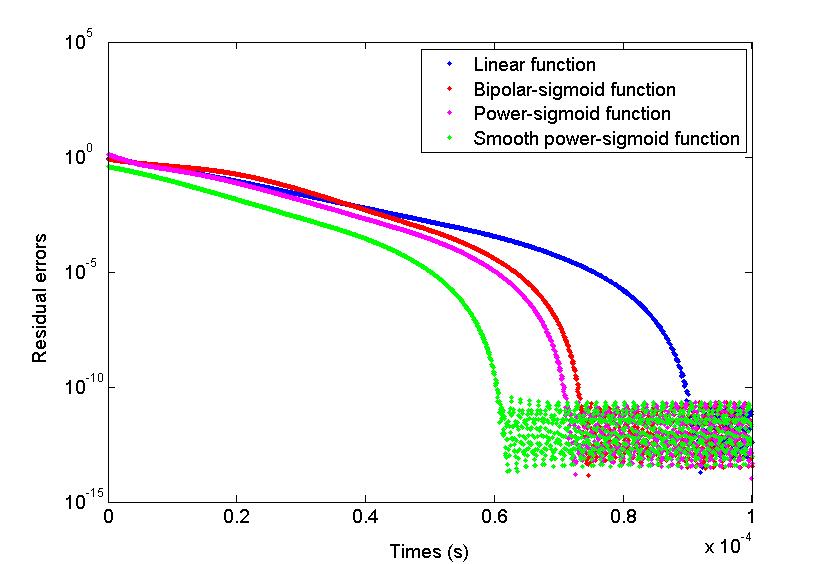}}\\
\subfigure[$\gamma=1000$, $q=(1,-2,3)^{\top}$.]{\includegraphics[width=3in, height=2.2in]{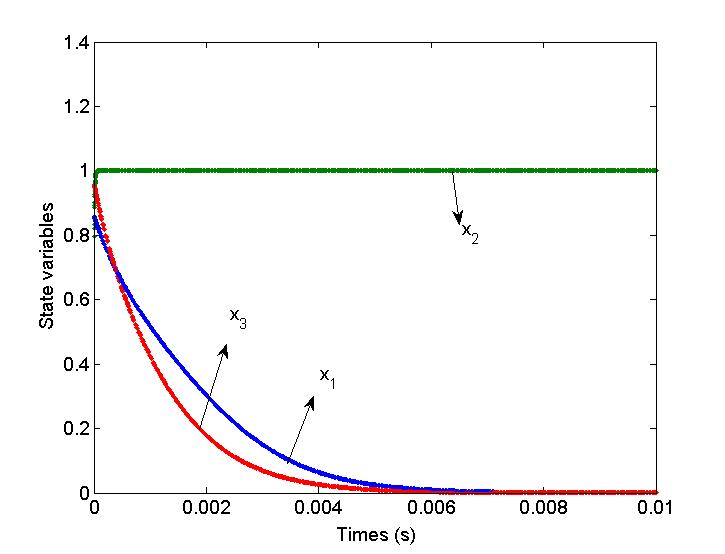}}
\subfigure[$\gamma=1\times 10^6$, $q=(1,-2,3)^{\top}$.]{\includegraphics[width=3in, height=2.2in]{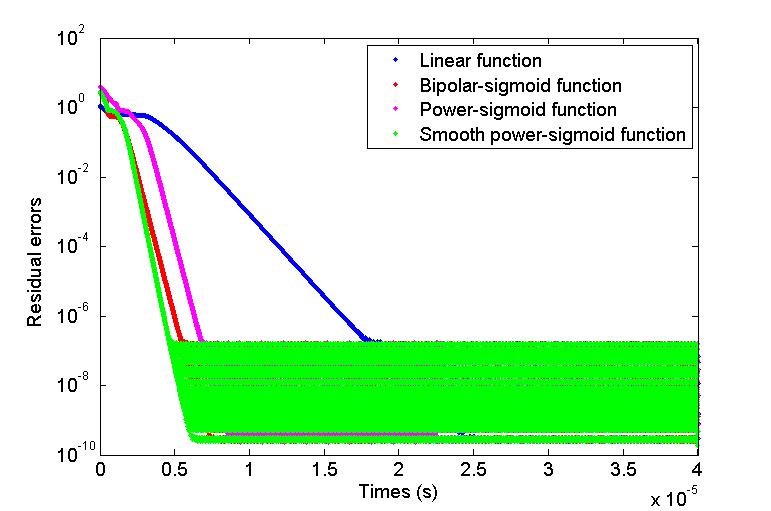}}\\
\subfigure[$\gamma=100$, $q=(-3,-2,-3)^{\top}$.]{\includegraphics[width=3in, height=2.2in]{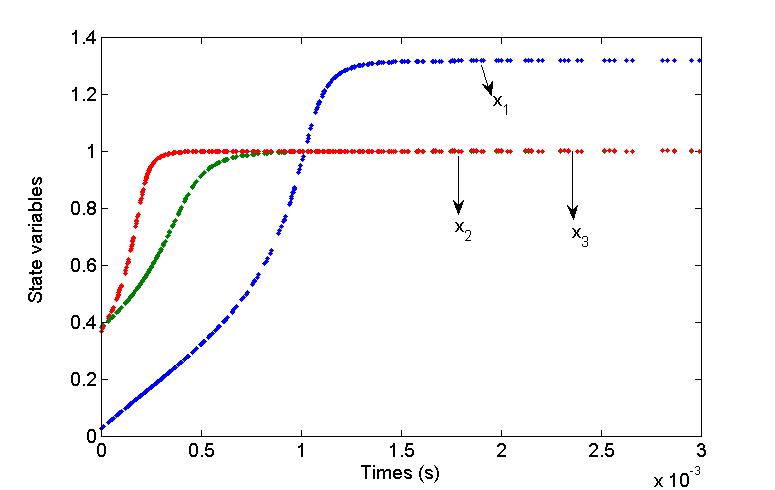}}
\subfigure[$\gamma=1\times 10^6$, $q=(-3,-2,-3)^{\top}$.]{\includegraphics[width=3in, height=2.2in]{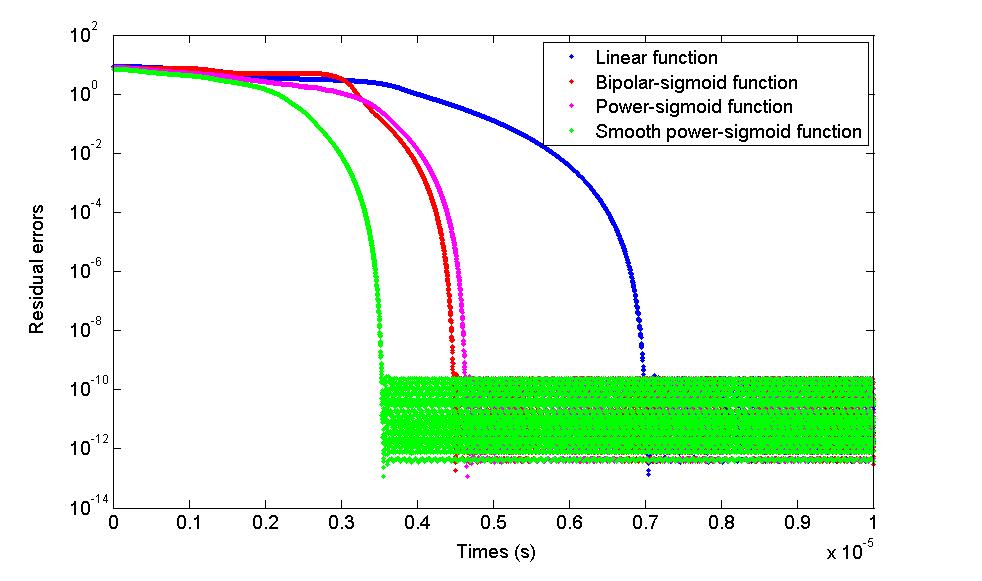}}
\caption{Trajectories of state variables and residual errors of the NGDSs for Example \ref{eg2}.}\label{figure 3}
\end{figure}
\begin{figure}[H]
\centering
\subfigure[$\gamma=100$, $q=(3,3,3)^{\top}$.]{\includegraphics[width=3in, height=2.2in]{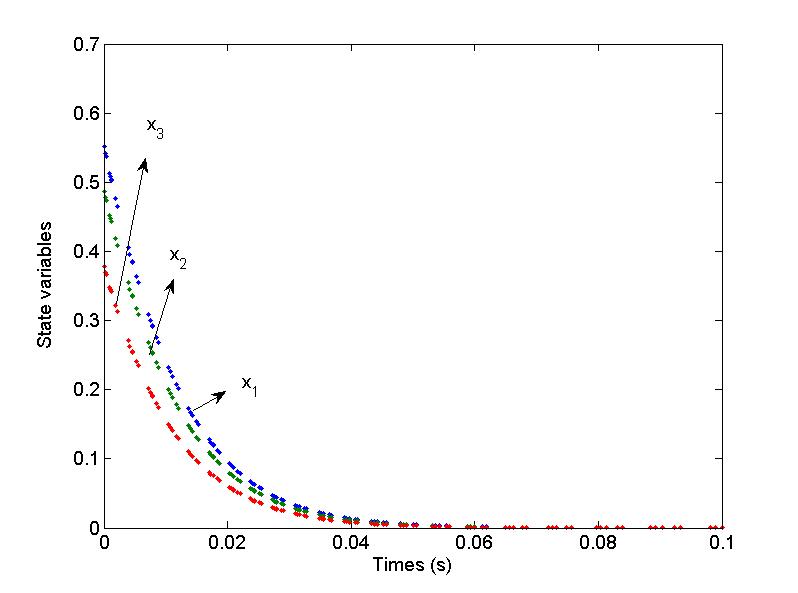}}
\subfigure[$\gamma=1\times 10^6$, $q=(3,3,3)^{\top}$.]{\includegraphics[width=3in, height=2.2in]{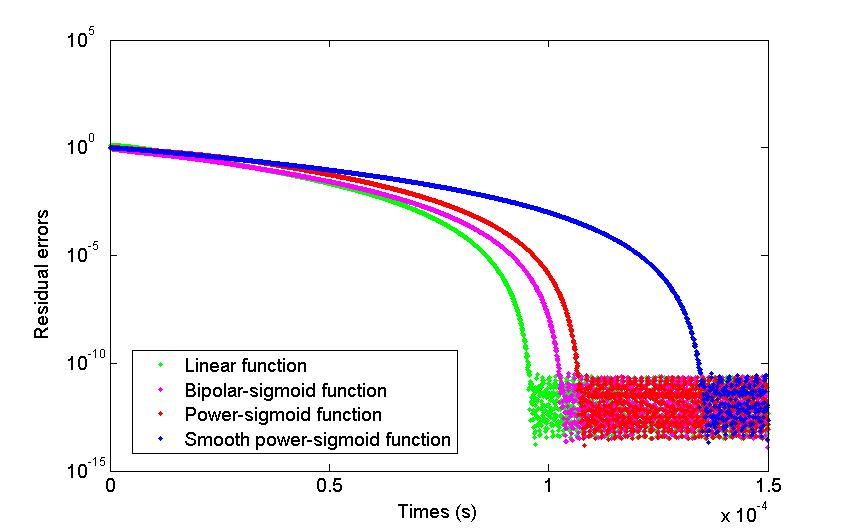}}\\
\subfigure[$\gamma=100$, $q=(-3,-1,-2)^{\top}$.]{\includegraphics[width=3in, height=2.2in]{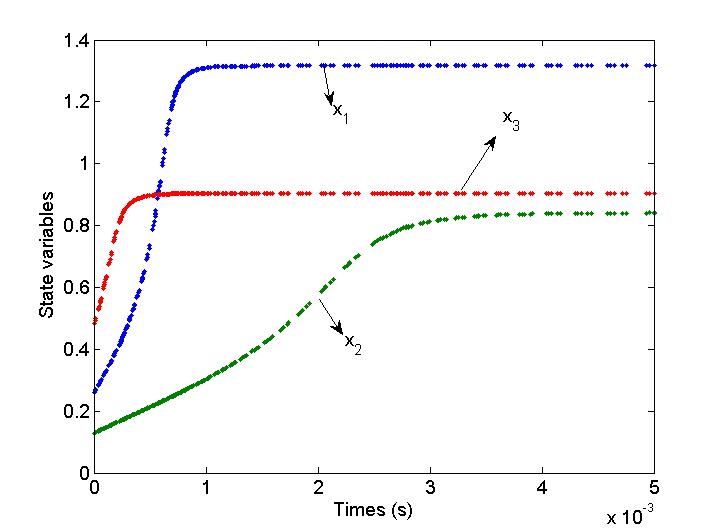}}
\subfigure[$\gamma=1\times 10^6$, $q=(-3,-1,-2)^{\top}$.]{\includegraphics[width=3in, height=2.2in]{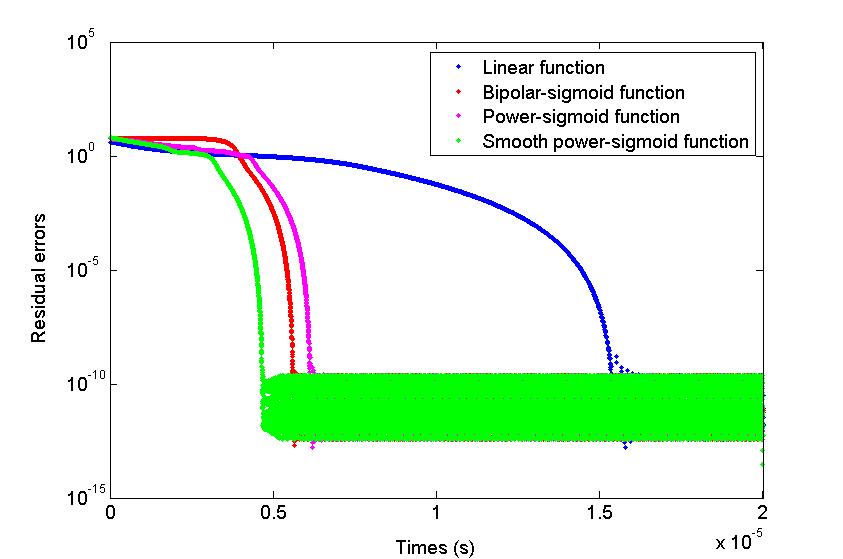}}\\
\subfigure[$\gamma=100$, $q=(-1,-1,-2)^{\top}$.]{\includegraphics[width=3in, height=2.2in]{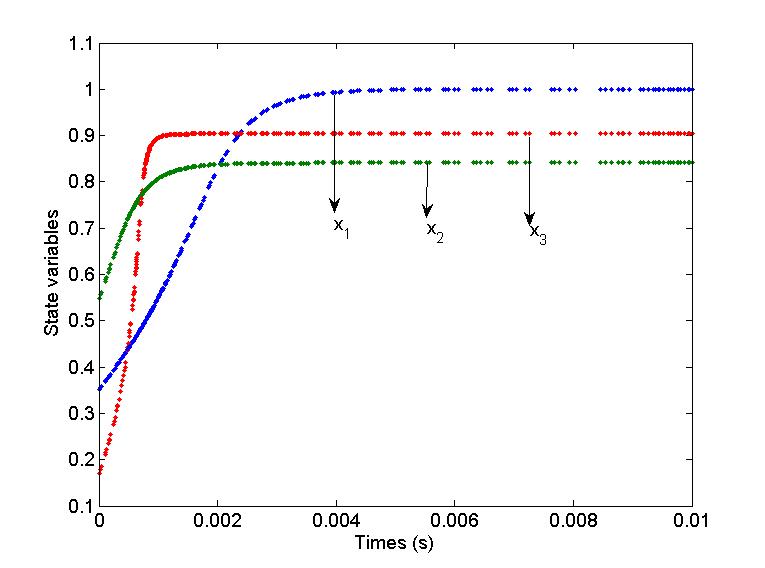}}
\subfigure[$\gamma=1\times 10^6$, $q=(-1,-1,-2)^{\top}$.]{\includegraphics[width=3in, height=2.2in]{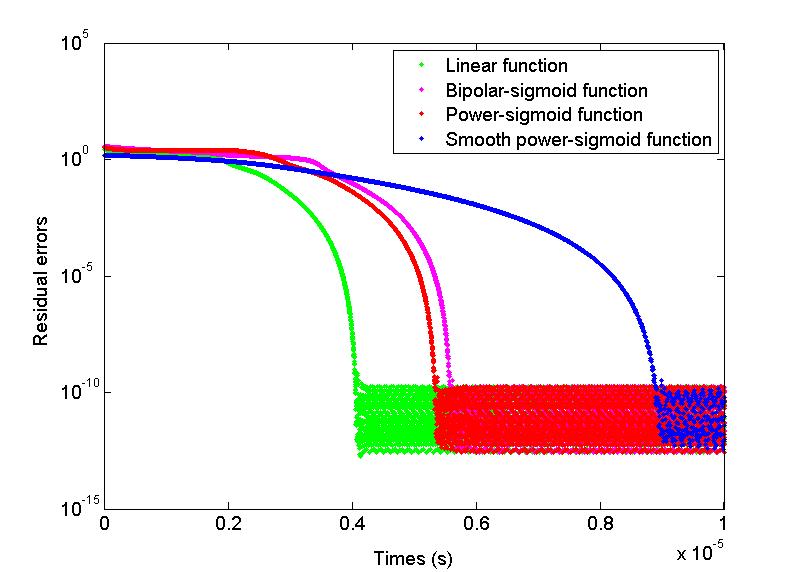}}
\caption{Trajectories of state variables and residual errors of the NGDSs for Example \ref{eg2}.}\label{figure 4}
\end{figure}

Where  blue stars indicate that $f(\cdot)$ is the linear function, red stars  correspond to the bipolar-sigmoid function $f_{bs}(x,7)$,
pink stars denote that power-sigmoid function $f_{ps}(x,5,7)$  and green stars are generated using $f_{sps}(x,5,9)$ as the smooth power-sigmoid function, respectively.

\begin{exmple} \label{eg3}
Let $\mathcal{A}\in \mathbb{R}^{[4,2]}$ be defined by
\[a_{1111}=1,\;a_{1112}=-2,\;a_{1122}=1,\;a_{2222}=1,\]
and all other $a_{i_1i_2i_3i_4}=0$. This tensor is  a $P$-tensor \cite{Bai2016Global}.
\end{exmple}
Taking $q=(0,-1)^{\top}$, it is easy to see that $x_{*}=(0,1)^{\top}$ and $x_{*}=(1,1)^{\top}$ are the solutions to the TCP$(\mathcal{A},q)$.

We choose  initial vector as $x_0=(0.1,0.5)^{\top}$ for $x_{*}=(0,1)^{\top}$ and $x_0=(1.5,1.1)^{\top}$ for $x_{*}=(1,1)^{\top}$. Trajectories of state variables corresponding to the LGDS with  $\gamma=100$ and  $\gamma=10000$ are shown in Figure \ref{figure 5} (a) and Figure \ref{figure 6} (a), respectively.

Residual errors (\ref{eq6.2}) derived by employing the NGDS with $\gamma=1\times 10^6$  shown in Figures \ref{figure 5} (b) and Figures \ref{figure 6} (b), where  blue stars indicate that $f(\cdot)$ is the linear function, red stars  correspond to the bipolar-sigmoid function $f_{bs}(x,7)$,
pink stars denote that power-sigmoid function $f_{ps}(x,5,9)$  and green stars are generated using $f_{sps}(x,7,11)$ as the smooth power-sigmoid function, respectively.

\begin{figure}[H]
\centering
\subfigure[$\gamma=100$, $x_{*}=(0,1)^{\top}$.]{\includegraphics[width=3in, height=2.2in]{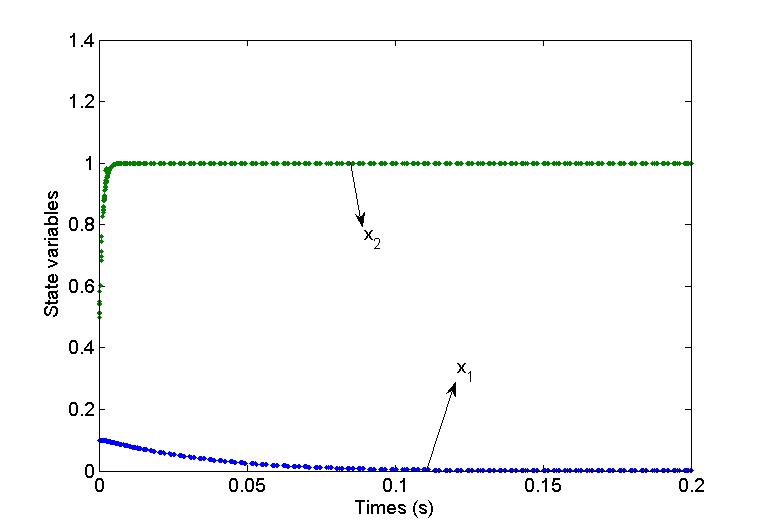}}
\subfigure[$\gamma=1\times 10^6$, $x_{*}=(0,1)^{\top}$.]{\includegraphics[width=3in, height=2.2in]{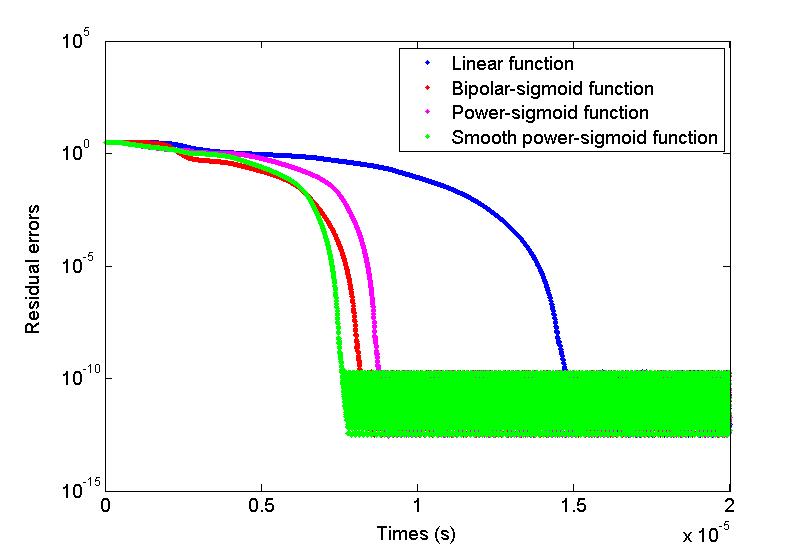}}
\caption{Trajectories of state variables and residual errors of the NGDSs for Example \ref{eg3}.}\label{figure 5}
\end{figure}

\begin{figure}[h]
\centering
\subfigure[$\gamma=10000$, $x_{*}=(1,1)^{\top}$.]{\includegraphics[width=3in, height=2.2in]{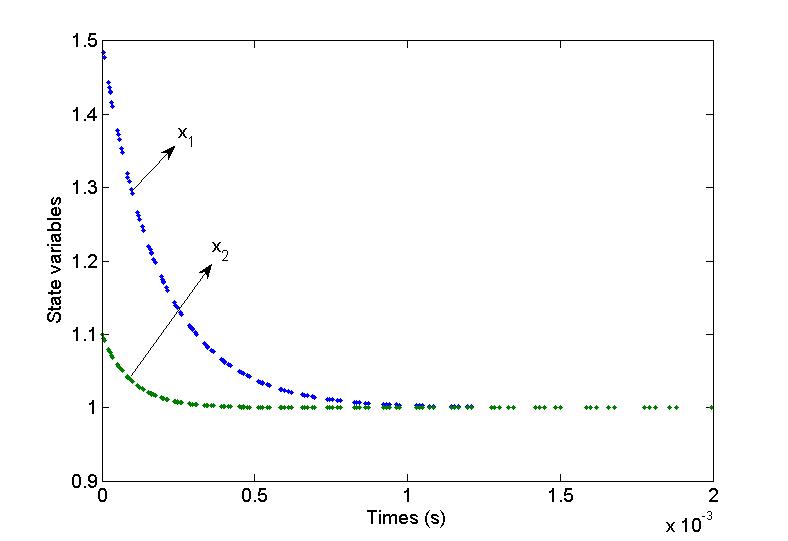}}
\subfigure[$\gamma=1\times 10^6$, $x_{*}=(1,1)^{\top}$.]{\includegraphics[width=3in, height=2.2in]{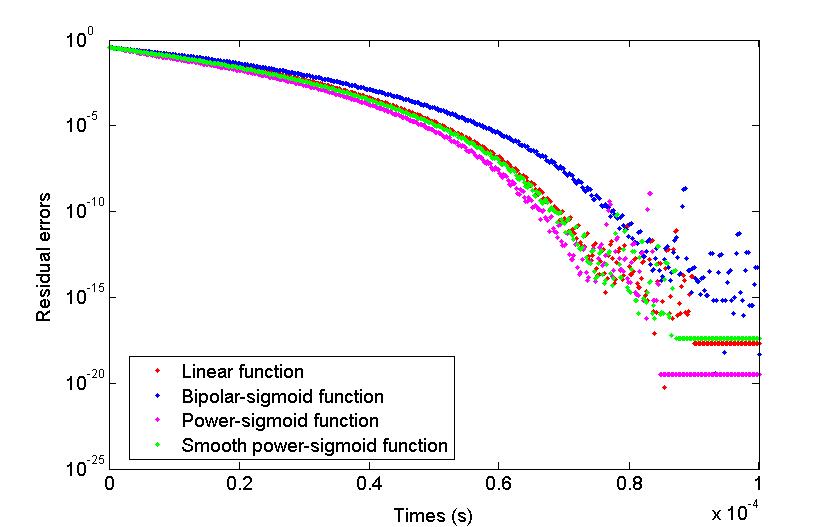}}
\caption{Trajectories of state variables and residual errors of the NGDSs for Example \ref{eg3}.}\label{figure 6}
\end{figure}

From the computer simulation results derived in the illustrative examples, the following conclusions can be highlighted:

1. The NGDS models with different types of activation functions  presented in section 3, exactly and efficiently solve the TCP$({\mathcal A},q)$ with $P$-tensors, strong $P$-tensor and strong strictly semi-positive tensor.

2. NGDS models could achieve different performances if different activation function arrays are used. In general, the convergence
performance of nonlinear activation functions is superior to that of the
linear activation function.

\end{document}